\documentclass[10pt,leqno]{amsart}
\usepackage{graphicx}
\usepackage{indentfirst,csquotes}

\usepackage{amssymb,amsthm,amsmath}
\usepackage{xcolor,paralist,hyperref,titlesec,fancyhdr,etoolbox}

\newtheorem{definition}{Definition}[section]
\newtheorem{theorem}{Theorem}[section]
\newtheorem{lemma}[theorem]{Lemma}
\newtheorem{corollary}{Corollary}[section] 
\newtheorem*{unnumberedtheorem}{Theorem}
\numberwithin{equation}{section}

\usepackage[sort, numbers]{natbib}

\titleformat{\section}[hang]{\normalfont\bfseries\LARGE}{\thesection}{1em}{}
\titlespacing{\section}{0pt}{\baselineskip}{0.5\baselineskip}

\hypersetup{ colorlinks=true, linkcolor=black, filecolor=black, urlcolor=black }

\usepackage{lipsum}

\title{Speed of Weak mixing for the Chacon map}
\author{Nelson Moll }
\date{July 2023}

\begin{document}

\maketitle

\begin{abstract}
We first consider a non-primitive substitution subshift that is conjugate to the Chacon map.  We then derive spectral estimates for a particular subshift and the speed of weak mixing for a class of observables with certain regularity conditions.  After, we use these results to find the speed of weak mixing for the Chacon map on the interval and show that this bound is essentially sharp.
\end{abstract}

 \section{Introduction}
 
\quad The Chacon map \cite{chacon1969weakly} is one of the first examples of a transformation that is weakly mixing but not strongly mixing. In fact, this is the 'generic' case.  It was shown in \cite{2f8d2139-f1f2-383b-bbc0-f40be59acc66} that the set of weakly mixing transformations in the space of invertible measure preserving transformations is a dense $G_{\delta}$ set. On the other hand,
 in \cite{rohlin1948general} it was proved that the 'general' measure preserving transformation is not mixing.  Other examples of systems which are weakly mixing but not mixing have since been found.  For example, in \cite{346190e9-6441-3b24-b68e-9889fa445298} it was proved that almost all interval exchange transformations are weakly mixing but are never mixing \cite{Katok1980IntervalET}.     
 
Our goal in this paper is to determine the rate of weak mixing for the Chacon map using estimates on the spectral measures.  We will use this estimate along with a result from  \cite{park2023size} to establish the existence of a set of times for which the Chacon map fails to be strongly mixing.  Using a spectral analysis approach to effective weak mixing was used in \cite{avila2021quantitative} to find the speed of weak mixing for interval exchange transformations.  The same method was also used in \cite{bufetov2021holder} and \cite{forni2021twisted} to find estimates for translation flows.  The general ’recipe’ is to derive H\"older-type bounds on the spectral measures
from upper bounds on twisted ergodic sums,  which in turn are proved using some sort of quantitative 
Veech criterion \cite{c40c3f8e-93da-39bf-87a6-9077825325cd} to bound the twisted sums.  A general scheme to this approach for the case of substitutions can be found in \cite{bufetov2014modulus}.

The main results in this paper are the following theorems.

\begin{unnumberedtheorem}[Theorem A]

Let $U$ be the operator on $L^2([0,1])$ defined by $U(f)=f \circ C$, where $C$ is the Chacon map.  
Let $f \in Lip([0,1])$ have zero expectation  and $g \in L^2([0,1])$.  Denote $\|f\|_L$ to be the Lipschitz norm of $f$ and $\| g\|_2$ the $L^2$ norm of $g$.  Then there is a constant $K_C>0$ dependent only on the Chacon map such that

\begin{center}
    $\frac{1}{N} \sum \limits_{k=0}^{N-1}|\langle U^kf, g \rangle |^2 \leq K_C \|f \|_L \| f\|_2 \|g \|_2^2 [\log_{3}N ]^{-\frac{1}{6}}$ \quad for all $N > 1$.
\end{center}
\end{unnumberedtheorem}
We also prove that this estimate is essentially sharp in Section 12.  By using the Cauchy-Schwartz inequality for sums with one sequence identically equal to one, we can compare the upper bound from Theorem A with the lower bound found in Theorem B.
 \begin{unnumberedtheorem}[Theorem B]

There exists a diverging sequence of $N$ and a constant $C > 0$ dependent only on the Chacon map such that, for each $N$ in the sequence, there is a Lipschitz $f_N$ and  a square integrable $g_N$ with the property that 
\begin{align*}
\sum \limits_{i=0}^{N-1} \left|\int_X f_N(T^ix) g_N(x) dx - \int_X f_N(x) dx \int_X g_N(x) dx\right| \\
& \geq  C \frac{N}{\log(N)^2} \cdot \| f_N\|_L^{\frac{1}{2}} \cdot \| f_N\|_2^{\frac{1}{2}} \cdot \| g_{N}\|_2.
\end{align*}

\end{unnumberedtheorem}

Notice that the speed of weak mixing for the Chacon map is logarithmic. Heuristically, this is in part due to the fact that the second highest eigenvalue for the substitution matrix of the Chacon substitution has norm equal to one.  In the case of interval exchange transformations of rotation type, it was proved in \cite{avila2021quantitative} that the speed of correlation decay is also logarithmic.    Contrast this with Section 7 of \cite{avila2021quantitative}, where it was shown that an interval exchange transformation that is not of rotation type has polynomial decay.  

 In the case of the Chacon map and rotation type IET's, we are able to bound the renormalization dynamics from the Veech argument (Section 6 of \cite{avila2021quantitative} for example) away from the integer lattice with some uniform frequency due to that fact that the second highest eigenvalue has norm 1.  In the case of interval exchange transformations, the case is slightly more subtle.  Rauzy-Veech induction on the interval exchange transformations gives us a family of substitutions to consider, and if the IET is of rotation type, then we can, in particular, use the rotation-like properties of the family of substitution matrices to gain logarithmic lower bounds.  This is explicitly found in Section 7.2 of \cite{avila2021quantitative}.  

For the case of substitutions, we can guess that the speed of weak mixing will be polynomial when the second highest eigenvalue of the substitution matrix is outside the unit circle.  For example, in Section 5 of \cite{bufetov2014modulus} it was shown that in an analogous case the speed of weak mixing is polynomial for suspension flows.  In general, there is a close relationship between translation flows, substitutions, and interval exchange transformations.  For reference, these relations are highlighted in \cite{avila2021quantitative}, \cite{bufetov2014modulus}, \cite{forni2021twisted} and \cite{bufetov2021holder}.  Although the paper shows the speed of weak mixing for the Chacon map, the end of Section 6 describes which technical assumptions are sufficient to extend the proof of effective weak mixing to other substitutions.  Apparently, there is a relationship between return words for substitutions and weak mixing that is still being investigated.  Knowledge of 'good' return words is almost equivalent to knowledge of the speed of weak mixing for a given substitution given other mild assumptions.

\section{Structure of the Paper}
  A concise introduction to substitutions and the Chacon map is found in sections 3 and 4 respectively.  The method we used to extract quantitative estimates on the substitution system comes from an analysis on the spectral measure of small intervals.  The relationship between twisted sums, spectral measure, and the quantitative rate of weak mixing is found in section 5.  The work in sections 6 through 8 serve to first find the rate of weak mixing for cylindrical functions defined on arbitrarily large cylinders.  Sections 10 and 11 use the fact that Lipschitz functions can be approximated by cylindrical functions to get the final upper bound for the speed of weak mixing.  What follows is a proof showing that this bound is essentially sharp.  Speed of weak mixing combined with a result from \cite{park2023size} also gives a corollary that bounds the frequency of times for which the Chacon map fails to be mixing. 
  
The Chacon map is constructed by cutting and stacking the unit interval, as outlined in \cite{chacon1969weakly}.  By tracking which intervals a particular point in the unit interval enters after iterating that point under the action of the Chacon map, we can assign almost every point in the unit interval a code given by its orbit.  This coding is conjugate to a primitive substitution subshift.  Bufetov and Solomyak showed in \cite{bufetov2014modulus} that it suffices to bound spectral measures and a Fourier anologue of the Birkhoff sums to find effective weak mixing.  In this setting, we can use the techniques from \cite{bufetov2014modulus} to find the speed of weak mixing for the Chacon substitution subshift.  This formulation is outlined in section 5.  In other terms, this methodology transfers the problem of effective weak mixing for $L^2$ functions to that of finding upper bounds for the spectral measures and the twisted sums.  To find quantitative bounds for the spectral measures, we will need both the speed of ergodicity, provided in \cite{adamczewski2004symbolic}, and a quantitative Veech argument similar to what is found in \cite{avila2021quantitative} and \cite{bufetov2014modulus}.

In this paper we utilize and generalize  results from \cite{bufetov2014modulus}, specifically estimates for the spectral measure of a function in terms of the twisted sums, along with quantitative results for rank 1 cylindrical functions. Expanding on these findings, we extend the results to cylindrical functions of arbitrary rank and demonstrate that, in the case of the Chacon map, we can derive explicit bounds for the twisted sums. These bounds, in turn, enable us to estimate the speed of weak mixing.  In particular we prove quantitative weak mixing for a primitive substitutions that is conjugate to the Chacon map.  We will then transfer these results back to the Chacon map on the unit interval using the map that codes the orbit of each point. This leads to the following theorem, proved in section 10. 

We will first prove effective weak mixing for simple functions on the substitution subshift using a Veech argument in sections 6 and 7 along with an estimate for the sped of ergodicity in section 8, then transfer these weak mixing estimates to a larger class of functions that can be approximated by these simple functions.  In order to expand our results from simple functions to general observables, we will apply the Chacon substitution to subwords that define the cylinders the simple functions are defined on.  We can get the growth rate of the twisted sums on simple functions determined by finite words, then leverage these estimates to get the result for infinite strings in the subshift.  This is done by applying the substitution to the finite words iteratively and examining how a spectral variation of simple functions on the substitution subshift grows as we apply the substitution iteratively to words on the subshift.  The explicit proof can be found in Section 6.  

The substitution matrix $S$ for the Chacon substitution has the crucial property that the second largest eigenvalue has norm equal to one.  Our estimate requires us to be able to bound vectors $v$ away from the integer lattice uniformly after iteratively applying the substitution matrix.  In particular, we need $\|S^n (\omega \cdot v) \|_{\mathbb{Z}} \geq |\omega|$ for some $\omega \in (0,1)$ for some uniform frequency  of $n \in \mathbb{Z}$ with $n \in [0,N]$.  This will give us a quantitative Veech estimate, which allows us to follow the type of argument found in \cite{bufetov2014modulus} or \cite{avila2021quantitative} to find effective weak mixing for the Chacon map.  

The work from \cite{park2023size} proves that there is an interval $A \subset [0,1]$ such that the number of  $n \in [0,N] \cap \mathbb{Z}$ such that $C^{-n}A \cap A = \emptyset$ is bounded below by a power of $\log(N)$.  Hence we can create functions $f_A$ supported on $A$ that by definition satisfy $\langle f_A \circ C^n, f_A \rangle = 0$ for a number of times $n$ bounded below in terms of $N$.  This allows us to derive lower bounds for effective weak mixing.  

\section{Construction of the Chacon Map }

The construction of the Chacon map and a proof of its properties can be found in \cite{chacon1969weakly}.  It was proved that the map is invertible, preserves the Lebesgue measure, and is weakly mixing but not mixing.  The Chacon map is created inductively by cutting subintervals out of $[0,1]$ and mapping them by translation onto each other by some inductive procedure.  This mapping can be visually represented by stacking the subintervals in a tower such that each level is mapped onto the one vertically above it by translation.  Each step in the procedure uses only part of the unit interval, but eventually every point in $[0,1]$ will included in the domain of the Chacon map.  The inductive scheme is as follows: \\

Step 1:  Take the unit interval and cut it into two pieces so that the left hand interval has length $\frac{2}{3}$ and the right hand interval has length $\frac{1}{3}$.  Call these pieces $I_0$ and $I_1$ respectively.  Now cut $I_0$ into three equal pieces, labeled first to third from left to right, and also cut an interval of length equal to that of the first interval in $I_0$ from $I_1$.  By translation we map the first interval from $I_0$ onto the second, the second interval onto the interval cut from $I_1$, and then map that interval onto the third interval from $I_0$.  We can pictorially represent these translations by stacking the range of the translation on top of its domain.  The image of each point is then the point vertically above it in the stack.   \\

Step $n$:  For the $n$'th step we cut the stack into three equally sized pieces and cut, starting from the left, an interval of equal size from the unused part of $I_1$.  By translation we map the top of the left stack to the bottom of the middle stack, send the top of the middle stack to the newly cut piece of equal length from $I_1$, and map this interval onto the bottom of the third stack.  Call the map that results $C_n$. 

\-\\
Now define the Chacon map by the pointwise limit $C = \lim \limits_{n \to \infty} C_n$.  Notice that for $m < n$ we have that $C_m(x) = C_n(x)$ on the domain of definition for $C_m$.  Since this domain of definition for $C_m$ is eventually  entire interval (mod sets of measure zero), the pointwise limit is defined almost everywhere.  \\

If we code the orbit of each point $x \in I$ by $x_i = j$ if $C^i(x) \in I_j$ then a pattern emerges in the string $\{ x_j \}$.  We can see that each word in $\{ x_j \}$ is contained in the string given recursively by $S_{n+1}=S_nS_n1S_n$ with $S_0=0$.  This recurrence can be modeled by a substitution, and it is from this model that we will be able to find quantitative estimates for the Chacon map.  Let $T$ be the shift map defined on $(\{0,1 \}^{\mathbb{Z}}, T)$, and let $\alpha$ be the substitution defined by

\begin{center}
    $\alpha(0)=0010$ \\
    $\alpha(1)=1$.
\end{center}

There is a measurable map $h$ with a measurable inverse from the unit interval to the substitution subshift in $(\{0,1 \}^{\mathbb{Z}}, T)$, called $(X,T)$, such that $h \circ C = T \circ h$.
This map can be defined explicitly by looking at the code of each $x \in I$ under the action of the Chacon map $(C, I)$.  Indeed, define $h: I \to X$ by $h(x)_i =j$ if $C^i(x) \in I_j$ where $j \in \{0,1 \}$.

\section{Substitutions}
Let $T: A^{\mathbb{Z}} \to A^{\mathbb{Z}}$ be the shift map on an alphabet $A$.  Define $A_w$ to be the set of words with letters in $A$.  A substitution is a map $\beta: A \to A_w$ and the substitution dynamical system is the shift map on a subshift defined to be the set of strings with the property that any word in the string is contained in a word of the form $\beta^n(a)$ for some $n \geq 0$ and $a \in A$.  \\

The substitution matrix $S$ is defined such that $S_{i,j}$ is equal to the number of letters $i$ in the word $\beta(j)$.  Notice that the length of the word $\beta^n(i)$, denoted by $|\beta^n(i)|$, satisfies 
\begin{center}
    $|\beta^n(i)|=\langle \hat{1}, S^n e_i \rangle$,
\end{center}
where $\hat{1}$ is the vector with entries all equal to $1$.
  We say that a substitution is primitive if the matrix $S$ is primitive, meaning there is an $n$ such that $S^n$ is entrywise positive.  A substitution subshift is a closed shift invariant set of $A^{\mathbb{N}}$.  If the substitution is primitive, then the shift is uniquely ergodic \cite{queffelec2010substitution}.   
  
The substitution given by 

\begin{center}
    $\alpha(0)=0010$ \\
     $\alpha(1)=1$
\end{center}

is conjugate  to the primitive substitution 

\begin{center}
    $\beta(0)=0012$ \\
    $\beta(1)=12$ \\
    $\beta(2)=012$.
\end{center}

You can go from $\beta$ to $\alpha$ by replacing every $2$ with a $0$.  The inverse is the function that sends every $0$ that appears just after a $1$ to a $2$.  Since primitive substitutions are uniquely ergodic the Chacon map is uniquely ergodic because it is conjugate to the primitive substitution subshift defined by $\beta$.

\section{Spectral Measures and the Speed of Weak Mixing}
We will first define quantitatively what it means for a process to "mix" two different sets together.  Let $T : X \to X$ be a measurable transformation that preserves the probability measure $\mu$ on $X$. 
\begin{definition}

A transformation $T: X \to X$ is called weakly mixing if for any $f,g \in L^2(X)$ we have 

\begin{center}
    $ \frac{1}{N} \sum \limits_{k=0}^{N-1}|\int_X f(T^kx)\overline{g(x)} d \mu-\int_X f d \mu \overline{{\int_X}g d \mu} | \to 0$ as $N \to \infty$
\end{center}
\end{definition}
Notice that in the definition of weakly mixing we only demand that the terms inside the sum converge to zero in average.  A strengthening of this condition is to demand that the sequence converges to zero. 
\begin{definition}
A transformation $T:X \to X $ is strongly mixing, or mixing, if for any $f,g \in L^2(X)$ 

\begin{center}
    $|\int_X f(T^kx)\overline{g(x)} d \mu-\int_X f d \mu \overline{\int_X g \, d\mu} | \to 0$ as $N \to \infty$
\end{center}
\end{definition}
 In this paper  we want to find some function $h(N)$ with the property $\frac{h(N)}{N} \to 0$ such that for any measurable functions $f$ and $g$ with certain regularity conditions we have 

\begin{center}
    $  \sum \limits_{k=0}^{N-1}|\int_X f(T^kx)\overline{g(x)} d \mu-\int_X f d \mu 
 \overline{\int_Xg d \mu} | =O(h(N))$.
\end{center}

In our search for  a suitable $h(N)$ we will make use of the spectral measure to simplify the sum that appears in the definition of weak mixing. \\

Let $(X,T,\mu)$ is a measure-preserving transformation.  We can associate each  $f \in L^2(X)$ with a positive measure on the unit circle, denoted $\sigma_{f,f}$ or $\sigma_{f}$, that has Fourier coefficients

\begin{center}
    $\hat{\sigma}_{f,f}(-k)=\int_0^1 e^{2 \pi i k \omega} d \sigma_{f,f}(\omega) = \langle f \circ T^k, f \rangle $.
\end{center}
The existence of such a measure is guaranteed by Bochner's theorem.  Let $U$ be the operator on $L^2(X)$ defined by $U(f)=f \circ T$.  Using the spectral measure we can reduce the sum found in the definition of weak mixing.  Indeed, for $f$ with zero average we have

\begin{align*}
    I_N= \sum \limits_{n=0}^{N-1} |\langle U^n(f), f \rangle |^2 \\
  &  =\sum \limits_{n=0}^{N-1} \langle U^n(f), f \rangle \int_{R /Z} e^{2 \pi i n \omega } d \sigma_f( \omega) \\
   & = \int_{R / Z} \langle \sum \limits_{n=0}^{N-1} e^{2 \pi i n \omega} U^n(f), f \rangle d \sigma_f(\omega) \\
& =I_{\epsilon}^{-}+I_{\epsilon}^{+}
\end{align*}

where 

\begin{equation}
I_{\epsilon}^{-}(N)= \int \limits_{- \epsilon}^{\epsilon} \langle \sum \limits_{n=0}^{N-1} e^{2 \pi i n \omega} U^n(f), f \rangle d \sigma_f(\omega)
\label{5.1}
\end{equation}

and

\begin{equation}
    I_{\epsilon}^{+}(N)= \int \limits_{ \epsilon}^{1-\epsilon} \langle \sum \limits_{n=0}^{N-1} e^{2 \pi i n \omega} U^n(f), f \rangle d \sigma_f(\omega). 
\end{equation}

The term $\sum \limits_{n=0}^{N-1} e^{2 \pi i n \omega} U^n(f)(x)=S_N^x(f, \omega)$ is called the twisted sum.  We can see from this estimate that we will need to find a bound for the twisted sum in order to find the speed of weak mixing.  Sections 6 and 7 provide the machinery to estimate $I_{\epsilon}^{+}$ and section 8 will estimate $I_{\epsilon}^{-}$.

\section{Estimating the Twisted Sums Using Matrix Products}

We say that $w$ is a word in the substitution if it is contained in $\beta^n(a)$ for some $n$ and letter $a \in A$.  Let $T$ be the shift map.  Let $w=w_1 \dots w_p$ be a word with $w_i \in A$.  We define $T^n(w)=w_{n+1} \dots w_p$ for $0 \leq n < p$. A cylinder of rank $n$ is defined to be the set of strings such that the first $n$ coordinates are equal to a fixed word of length $n$.   

For $x \in X$, define $x[0,k]=x_0 \dots x_{k}$.  A cylinder of rank $n$ is defined to be the set of $x \in X$ such that $x[0,n-1] = w$ for some word $w$ in the substitution with length $n$. Let $J_n$ be the number of cylinders of length $n$ in the shift invariant substitution space $X$.  We will enumerate each cylinder of rank $n$ and index them with $k$.  The notation $[k,n]$ means the $k$'th cylinder of rank $n$ and $1_{[k,n]}$ is the characteristic function on $[k,n]$.  Since $\beta$ is a primitive substitution we have $J_n \leq c'n$ for some $c'>0$ that depends only on the substitution \cite{queffelec2010substitution}.   If a function $f: X \to \mathbb{R}$ is of the form

 \begin{center}
     $f(x)= \sum \limits_{k=1}^{J_n} r_k 1_{[k,n]}(x)$ \quad for $r_k \in \mathbb{R}$
 \end{center}
then we say that $f$ is a cylindrical function of rank $n$.  Note that $n$ is the length of the cylinder $[k,n]$.  \\

  The following results in this section follow the strategy  used in section 3 of \cite{bufetov2014modulus}.  We will alter that method slightly to prove resutls for cylindrical functions of arbitrarily large rank.  Let $v$ be a word of length equal to or greater than $n$.  
\begin{definition} For a word $v$ in the substitution and $\omega \in [0,1)$ let
\begin{center}
    $\phi_{[k,n]}(v,\omega)=\sum_{j=0}^{|v|-n} 1_{[k,n]}(T^jv)e^{-2 \pi i \omega j}$.
\end{center}
\end{definition}

There is a formula for $\phi_{[k,n]}(vw, \omega)$ in terms of $\phi_{[k,n]}(v,\omega)$ and $\phi_{[k,n]}(w,\omega)$ that is similar to what is found in \cite{bufetov2014modulus} for rank one cylindrical functions.  One notable difference is that the indicator function $1_{[k,n]}$ will eventually process words that are parts of both $v$ and $w$ after they have been shifted a sufficient number of times.  The following term will account for this.
\begin{definition}
Let $v=v_1 \cdots v_{p}$ and $w=w_1 \cdots w_q$ be words in the substitution with $v_i, w_i \in A$ and $p,q \geq n > 1$.  Define
\begin{center}
    $ H(v,w,\omega,n,k) = \sum \limits_{j=1}^{n-1}1_{[k,n]}(v_{p-j+1} \cdots v_{p}w_1 \cdots w_{n-j})e^{- 2 \pi i \omega j } $.
    \label{def5.2}
\end{center}
\end{definition}
The $\omega$, $n$ and $k$ will often be fixed and sometimes dropped.  
\begin{lemma} Let $v$ and $w$ be words in the substitution and $\omega \in [0,1)$.  Then
\begin{center}
     $\phi_{[k,n]}(vw,\omega)=\phi_{[k,n]}(v,\omega)+e^{-2 \pi i \omega |v|}\phi_{[k,n]}(w,\omega)+ e^{- 2 \pi i \omega (|v|-n+1) }H(v, w, \omega, n , k)$
\end{center}
\label{lemma6.1}
\end{lemma}

\begin{proof}
The proof can be seen by simply writing out the definition of the terms on the right hand side.  
\end{proof}

  For the remainder of the section, our objective will be to find a bound for entries in an associated matrix product using an inductive process in order to bound the $\phi_{[k,n]}$ .

  \begin{definition} Let $1, \dots ,p \in A$ enumerate the letters in the alphabet $A$ and let $v^t$ denote the transpose of the vector $v$.  For $\omega \in [0,1)$ we  define  
\begin{equation*}
 \Psi_{m}^{[k,n]}(\omega)=[\phi_{[k,n]}(\beta^m(1),\omega), \cdots ,\phi_{[k,n]}(\beta^m(p),\omega)]^t   
\end{equation*}
where $[k,n]$ is the $k$'th cylinder of rank $n$.
\end{definition}

As in the case of rank one cylinders, if we can estimate the growth of $\phi_{[k,n]}(\beta^m(i),\omega)$ then we can obtain bounds for the twisted Birkhoff sums for the cylindrical function $f$. This is due to the fact that the words $\beta^m(i)$ are in some sense dense in the substitution subshift.\\
  \\*
  Remark:   A technicality to keep in mind is we need the length of the words $\beta^m(i)$ to be greater than $n$, the length of the cylinder.  The Perron-Frobenius theorem tells us that there are constants $c'$ and $c$ such that $c'\theta^m \leq |\beta^m(i)| \leq c \theta^m$ for each $i \in A$, where $\theta$ is the Perron-Frobenius eigenvalue.  Hence we will assume that the fixed length of the cylinder $n$ satisfies $n \leq c' \theta^m \leq |\beta^m(i)|$. \\
  \\*
  
  In parallel with \cite{bufetov2014modulus} we begin our estimates by rewriting the relationship in Lemma \ref{lemma6.1} as a matrix product with some error term.  
\begin{definition} Let $\omega \in [0,1)$ and let the length of the substitution cylinders be a fixed $n$. Define 
 \begin{center}
    $\Pi_m(\omega) := [\Psi_m^{[1,n]}(\omega) \cdots \Psi_m^{[J_n,n]}(\omega)]$.
    \end{center}
\end{definition}

 \begin{lemma}  There is a matrix $M_{m-1}(\omega)$ and a vector $E_{m-1}(\omega)$ with $|E_{m-1}| \leq Kn$ entrywise such that for \quad $ K' + m \geq  \log_{\theta}(n)$ \quad and $\omega \in [0,1)$ we have the following relation:
 \begin{center}
    $\Pi_m(\omega)=M_{m-1}(\omega) \Pi_{m-1}(\omega)+E_{m-1}(\omega)$.
\end{center}
\label{lemma5.2}
\end{lemma}

\begin{proof}
  The result will follow after looking at the expansion  of $\phi_{[k,n]}(\beta^m(b), \omega)$.

Let $\beta(b)=u_1^{(b)} \cdots u_{k_b}^{(b)}$   with $u_i^{(b)} \in A$.  
 
\begin{align*}
    \phi_{[k,n]}(\beta^m(b), \omega) &= \sum_{j=1}^{k_b} \exp \left[[-2 \pi i \left(|\beta^{m-1}(u_1^{(b)})\omega| + \cdots + |\beta^{m-1}(u_{j-1}^{(b)})|  \right)\omega \right]\phi_{[k,n]}(\beta^{m-1}(u_j^{(b)}), \omega) \\
    & \quad + \sum_{j=1}^{k_b-1}e^{i \alpha_j} H \left(\beta^m(u_j^{(b)}) , \beta^m(u_{j+1}^{(b)}) \right),
\end{align*}

where the $\alpha_j \in \mathbb{R}$ is calculated as in Lemma \ref{lemma6.1} and $H$ is from Definition \ref{def5.2}.  This implies the following definition for $M_{m-1}(\omega)$:

\begin{equation*}
   M_{m-1}(\omega)_{b,c}=\sum \limits_{j \leq  |\beta(b)|:  u_j^{(b)}=c} \exp \left[-2 \pi i \left(|\beta^{m-1}(u_1^{(b)})| + \cdots + |\beta^{m-1}(u_{j-1}^{(b)})|  \right)\omega \right].
\end{equation*}
Now define 
\begin{equation*}
E_m = \left[
\sum_{j=1}^{k_1-1} e^{i \alpha_j} H \left(\beta^m(u_j^{(1)}), \beta^m(u_{j+1}^{(1)})\right), \cdots, \sum_{j=1}^{k_p-1} e^{i \alpha_j} H \left(\beta^m(u_j^{(p)}), \beta^m(u_{j+1}^{(p)})  \right)
\right]^t
\end{equation*}
where $p$ is the size of the alphabet for the substitution.   
\end{proof}

Hence for $m + K' \geq  \log_{\theta}(n)$ we can use induction on Lemma \ref{lemma5.2} to get the following formula.

\begin{equation}
 \Pi_m(\omega)= \left[ \prod \limits_{j=n+1}^m M_j(\omega) \right] \Pi_n(\omega) + \sum \limits_{k=n+1}^m M_k(\omega) E_{k-1}(\omega).
 \label{piece1}
\end{equation}

To proceed with bounding $\Pi_m(\omega)$ we need to bound both the $E_{m-1}$ and the product of the $M_{k}$.  The majority of the rest of this section will be dedicated towards finding a bound on the product of the $M_j(\omega)$.

\begin{theorem} There is a constant $0 < c' < 1$ dependent only on the substitution such that 
\begin{center}
$\left[ \prod \limits_{j=n+1}^m M_j(\omega) \right] \hat{1} \leq \prod \limits_{k=n+1}^{m} (1-c' \| \omega  \|^2) (S^t \hat{1})^{m-n}$.
\end{center}
entrywise.
\label{veech}
\end{theorem}

  Since  $\omega$ will be fixed, we  occasionally omit it in our notation.  Our proof will depend on a series of quantitative lemmas.  Furthermore, inequalities comparing matrices and vectors are understood to be entrywise.  We will also denote $|A|$ to be the matrix with entries equal to the absolute value of those from $A$.

\begin{lemma}
The entries of $\Pi_n(\omega)$ are bounded above by $c \theta^n$, where $c>0$ is dependent only on the substitution.
\end{lemma}

 \begin{proof}
  The matrix $\Pi_{n}(\omega)$ has entries equal to 

\begin{center}
    $\phi_{[k,n]}(v,\omega)=\sum_{j=0}^{|v|-n} 1_{[k,n]}(T^jv)e^{-2 \pi i \omega j}$
\end{center}

with $v=\beta^n(i)$.  Hence we have the entrywise bound 
\begin{equation*}
  |\phi_{[k,n]}(v,\omega)|\leq |v|-n \leq c \theta^n 
\end{equation*}
for $\Pi_n(\omega)$.

\end{proof}

\begin{lemma}
$|\sum \limits_{k=n}^m M_k(\omega) E_{k-1}(\omega)| \leq m n C_S$ where $C_S>0$ depends only on the substitution.
\label{piece2}
\end{lemma}

\begin{proof}
This follows from the fact that $|M_k(\omega)| \leq S^t$, where $S$ is the substitution matrix, and $|E_k(\omega)| \leq n$.  Now set $C_S = \| S^t\|$

\end{proof}

\begin{definition} Let $v $ be an $m$-dimensional real vector.  Define 

\begin{center}
$\|v \|_{\mathbb{Z}}=\max_{i} \|v_i \|_{\mathbb{Z}}$
\end{center}
to be the maximum of the distances between each coordinate $v_i$ to the integers.
\label{definitionreturn}
\end{definition}
We will sometimes abbreviate the notation to $\| v\|$.  Note that in the one dimensional case this reduces to just the distance of the point to the integers.

Our proof will need the existence of words $v_i$ such that for each $b \in A$ there is a word $v_i$ and a letter $c_i$ such that $v_i$ starts with $c_i$ and $v_ic_i$ is contained in $\beta(b)$.   We will call the $v_i$ return words.  Hence for each $b$ we can define $p_i^{(b)}$ and $q_i^{(b)}$ so that  $\beta(b)=p_i^{(b)}v_ic_iq_i^{(b)}$.  Note that the $M_n(\omega)$ have entries that are trigonometric polynomials with coefficients equal to $1$ with terms not more than the corresponding integer entry in the transpose of the substitution matrix $S^t$.    Using the definition of $M_n$
we have that the terms $e^{-2 \pi i |p_i^{(b)}| \omega}$ and $e^{-2 \pi i |p_i^{(b)}v_i|\omega}$ are both contained in the entry $M_n(\omega)_{b,c_i}$.  By excluding them through subtraction and then adding the absolute value of their sum we get 
\begin{center}
    $|M_n(\omega)_{b,c_i}| \leq (S^t)_{b,c_i}-2+|e^{2 \pi i \omega |\beta^n(v_i)|}+1|$.
\end{center}

The inequality  
\begin{center}
    $|1+e^{2 \pi i  t}| \leq 2 - \frac{1}{2}\|t\|^2$
\end{center}

then implies 

\begin{center}
    $|M_n(\omega)_{b,c_i}| \leq (S^t)_{b,c_i}-\frac{1}{2} \|\omega|\beta^n(v_i)| \|^2$
\end{center}

for each $i$.  \\

From  \cite{bufetov2014modulus} we get, for $\hat{x}=(x_1, \dots, x_n)>0$ a positive vector and $v$ a fixed return word that starts with the letter $c \in A$, 

\begin{equation}
\begin{aligned}
(|M_k|\hat{x})_b &= \sum_{j=1}^p|M_k(b,j)|x_j \\
&\leq \sum_{j=1}^p S^t(b,j)x_j - \frac{1}{2}\|\omega |\beta^k(v)|\|^2 x_c \\
&\leq (1-c(\hat{x})\|\omega |\beta^k(v)|\|^2) \sum_{j=1}^p S^t(b,j)x_j \\
&= (1-c(\hat{x})\|\omega |\beta^k(v)|\|^2) (S^t\hat{x})_b,
\end{aligned}
\label{productphase}
\end{equation}

where $c(\hat{x})$ is defined by the equation
\begin{equation}
c(\hat{x})=\frac{x_c}{2m \max_j (S^t)_{b,j} \max_j x_j}. 
\label{c(x)}
\end{equation}
 
 We can use this formula inductively to get a bound for a product of the $M_k(\omega)$ . During each step of the induction we can pick a suitable return word $v_k$  to control the dynamics of the quantity $\|\omega|\beta^k(v)| \|$.  The choice of $v_k$ will be determined in the next few paragraphs.  

Note that equation (\ref{productphase}) implies that entrywise
\begin{equation}
|M_{k+1}M_{k}| \hat{x} \leq |M_{k+1}| \left((1-c(\hat{x}) \|\omega |\beta^k(v)|\|^2)S^t \hat{x}  \right). 
\end{equation}
If we use this  along with the substitution $\hat{x} \to S^t \hat{x}$  we get the entrywise inequality

\begin{equation}
M_{k+1}M_{k} \hat{x} \leq \left(1-c(\hat{x}) \| \omega |\beta^k(v)| \|^2) (1-c(S^t\hat{x}) \| \omega |\beta^k(v)| \|^2) \right)(S^t)^2 \hat{x}.
\end{equation}

Now set $\hat{x} = \hat{1}$.  The Perron-Frobenius theorem implies that $ c=\inf_n \{c((S^t)^n \hat{1}) \}$ is positive.  Using the above calculations iteratively we obtain the bound 

\begin{equation}
   | \prod \limits_{j=n+1}^m M_j(\omega) | \hat{1} \leq \prod \limits_{k=n+1}^{m} (1-c \| \omega |\beta^k(v_k) | \|^2) (S^t \hat{1})^{m-n}. 
\end{equation}
\-\\

The following is found in \cite{bufetov2021holder}.  Since the substitution is primitive, the dynamical system $X_{\beta}$ is equal to $X_{\beta^n}$ for each $n \geq 1$.  Recall that return words were defined after Definition \ref{definitionreturn}.  It is possible that there are no return words when examining $\beta$, however, we can pass to $\beta^n$ and obtain the existence of return words if needed.  In the case of the Chacon substitution we note that all of $v_1=12$, $v_2=012$ and $v_3=01201$ are return words for $\beta^3$.  If $v$ is a word in the substitution we define the population vector $l(v)$ to be equal to the vector whose $i$'th entry is equal to number of the letter $i$ in $v$.  Notice that the population vectors for these return words generate $\mathbb{Z}^3$.

This implies that there are $a_{i,j} \in \mathbb{Z}$ such that $\sum_{i=1}^ka_{j,k}l(v_k)=e_j $.  Hence
\begin{equation}
\begin{split}
\| \hat{x} \| = \max_{i} \| x_i \|_{\mathbb{Z}} & = \max_i \| \langle \sum_{j=1}^{k} a_{i,j} l(v_j), \hat{x} \rangle \|_{\mathbb{Z}} \\
 &\leq \max_{i \leq m} \sum_{j=1}^k |a_{i,j}| \cdot \max_{j \leq k} | \langle l(v_j) , \hat{x} \rangle  | \\
\end{split}
\end{equation}
This shows that there is a $C>1$ independent of $x$ such that 
\begin{equation}
    C^{-1}\| \hat{x} \| \leq \max_{j \leq k} | \langle l(v_j) , \hat{x} \rangle  | 
\end{equation}

The Chacon map is conjugate to the substitution given by 
\begin{center}
 $0 \to 0012$ \\ $ 1 \to 12$ \\ $2 \to 012$
\end{center}
The substitution matrix $S$ is 
\begin{center}
$
    \begin{bmatrix}
    2 & 0 & 1 \\
    1 & 1 & 1 \\
    1 & 1 & 1
    \end{bmatrix}
    $
\end{center}

and  $b = [-1,1,1]^t$ is an eigenvector for $S$ with eigenvalue $1$.  Now define  

\begin{equation}
    \hat{x}_k(\omega) := \omega (S^t)^k \hat{1}.
\end{equation}

\begin{lemma}
There is an $\alpha  \in (0,1)$ such that for all $k \in \mathbb{Z}^+$ and all $\omega \in [0,1)$ we have  $\| \hat{x}_k(\omega)\|_{\mathbb{Z}} \geq \alpha \| \omega\|_{\mathbb{Z}}$
\end{lemma}
\begin{proof}
If not then there are sequences $k_j$ and $\omega_j > 0$ such that 

\begin{equation*}
\|\hat{x}_{k_j}(\omega) \| \leq \frac{1}{j} \| \omega_j\|.
\end{equation*}
   Note that $\langle \hat{x}_{k_j}(\omega), b \rangle= \omega$ since $\langle b, \hat{1}\rangle = 1$ and $Sb = b$.  This implies that 

\begin{equation*}
    \| \omega_j\|_{\mathbb{Z}} = \| \langle \hat{x}_{k_j}(\omega), b, \rangle \|_{\mathbb{Z}} \leq \left( \sum \limits_{i} |b_i| \right) \| x_{k_j}(\omega)\|_{\mathbb{Z}} \leq \frac{1}{j}  \left( \sum \limits_{i} |b_i| \right) \| \omega_j\|_{\mathbb{Z}}
\end{equation*}
for all $j$, a contradiction.
\end{proof}

It follows that 

\begin{equation}
   \| \omega (S^t)^k \hat{1}  \|_{\mathbb{Z}}= \| \hat{x}_k(\omega) \|_{\mathbb{Z}}   \geq \alpha \| \omega \|_{\mathbb{Z}} 
\end{equation}

for some $1 > \alpha > 0$.

Thus  we obtain 
\begin{equation}
\prod \limits_{k=n+1}^{m} (1-c(x_k) \| \omega |\beta^k(v_k) | \|^2) (S^t \hat{1})^{m-n} \leq \prod \limits_{k=n+1}^{m} (1-c' \| \omega  \|^2) (S^t \hat{1})^{m-n}.
\end{equation}
where $v_k \in \{12, 012, 01201\}$ is chosen to maximize $|\langle l(v_k), \hat{x}_k \rangle|$ and $c'=\alpha \cdot \inf_i \{c(x_i) \}>0$ by the Perron-Frobenius theorem.  
This completes the proof of Theorem \ref{veech}.  

\qedsymbol{} \\
 
\begin{theorem}
There is an independent constant $k>0$ such that for each $i$ and $j$
\begin{center}
    $(e_i)^t| \prod \limits_{j=n+1}^m M_j(\omega) |e_j \leq k \prod \limits_{j=n+1}^{m} (1-c' \| \omega  \|^2) \theta^{m-n}$
\end{center}
where $\{e_k \}$ is the standard basis.
\label{piece3}
\end{theorem}

\begin{proof}
This follows from Theorem 5.3 and the inequality $\langle \hat{e}_j, (S^t)^q \rangle = |\beta^q(j) | \leq k \theta^q$,  

where $k>0$ depends only on the substitution.
\end{proof}
 
\begin{corollary}
Let $m \leq K \log_{\theta}N$ for some constant $K>1$  dependent only on the substitution.  Then, entrywise, there are consants $c''$ and $C_S$ dependent only on the substitution such that
\begin{center}
   $|\Pi_m(\omega)| \leq C_S N^{1-c''\| \omega\|^2} +C_S mn$.
\end{center}
\label{corollary6.1}
\end{corollary}

\begin{proof}

Choose $n$ so that $n \leq \frac{1}{2} \log_{\theta}N$  Since the entries in the product are less than one, and since $m-n-1 \geq \frac{1}{2} \log_{\theta} N$ for large enough $N$, we have 

\begin{equation}
    \prod \limits_{k=n+1}^{m} (1-c' \| \omega  \|^2) \leq ((1-c' \| \omega  \|^2))^{\frac{1}{2} \log_{\theta}N} = N^{ \frac{1}{2}\log_{\theta}(1-c'\|\omega \|^2)} \leq N^{-c'' \| \omega\|^2}.
\end{equation}

The result then follows from Lemma \ref{piece1}, Lemma \ref{piece2}, Theorem \ref{veech}, and Theorem \ref{piece3}.

\end{proof}

The ingredients specific to the Chacon map to prove Theorem \ref{veech} include the existence of a set of return words whose population vectors span $\mathbb{Z}^n$. We also  used the property that the eigenvector for the substitution matrix corresponding to eigenvalue $1$ has non-trivial projection onto the vector with ones in all entries and that the substitution is primitive.  These assumptions are sufficient to extend the quantitative Veech criterion, or Theorem \ref{veech}, to other substitutions that satisfy the above conditions.    

\section{Estimates involving arbitrary strings}
In the previous section we found estimates for the twisted sums on words of the form $\beta^k(b)$ for some $b \in A$.  We will use this result to find bounds for the growth of twisted sums on arbitrary strings.  What is found in this section is partially a generalization of the techniques found in \cite{bufetov2014modulus}.  Let $[k,n]$ be the $k$'th cylinder of rank $n$.  We will consider cylindrical functions \
$f = \sum_k r_k 1_{[k,n]}$. Set

\begin{center}
    $\phi_f(v,\omega)= \sum_k r_k \phi_{[k,n]}(v,\omega)$
\end{center}

The following theorem is the building block of the paper.  Later we will see that cylindrical functions are able to approximate Lipschitz functions sufficiently enough to transfer the following bound.

\begin{theorem}

Let $n$ be the rank of the cylidnrical function $f$ and let $x[0,N-1]$ be the word comprised of the first $N$ terms of the string $x$.  Let $ n  \leq \frac{1}{2} \log_{\theta}N$ and $m \leq C  \log_{\theta}N$ for some $C > 0$ large enough. Then we have the following bound for the twisted sum:

\begin{equation*}
   |\phi_f(x[0,N-1], \omega)| 
    \leq 2L \cdot C_S' \cdot n \cdot  \| f\|_L \left( N^{1-c'\|\omega \|^2} 
    + 6(\log_{\theta}N)^2  +2 \log_{\theta}N + 1 \right)  
\end{equation*}

\label{theorem7.1}

\end{theorem}

The main idea behind the proof is that the words $\beta^k(b)$ form the building blocks of the strings $x \in X$. The following lemma, found in \cite{10.1007/978-1-4612-3352-7_22}, gives an explicit statement. \\
\begin{lemma} Let $x \in X$ and $N \geq 1$.  Then 
\begin{center}
     $x[0,N-1]=u_0 \beta(u_1) \cdots \beta^m(u_m) \beta^m(v_m) \cdots \beta(v_1)v_0$
\end{center}
where $m \geq 0$ and the $u_i, v_i$, are respectively proper prefixes and suffixes of the words $\beta(b)$ for $b \in A$.  
\label{lemma7.2}
\end{lemma}
The words $u_i$ and $v_i$ may be empty but at least one of $u_n, v_n$ is non-empty.  We will also need the following lemma, which relates the $N$ in $x[0,N-1]$ with the $m$ in the prefix-suffix decomposition.  This next result is a corollary of the above lemma and the Perron-Frobenius theorem.\\
\-\\
\begin{lemma}  For all $b \in A$ there are constants $c,c'>0$ such that 

\begin{center}
    $c \theta^m \leq \min \limits_{b \in A}|\beta^m(b)| \leq N \leq 2 \max \limits_{b \in A}|\beta^{m+1}(b)| \leq 2c' \theta^{m+1}$
\end{center}
\end{lemma}
We can now decompose the generalized twisted sum for $x[0,N-1]$ by using Lemma \ref{lemma6.1} and Lemma \ref{lemma7.2}.

\begin{equation}
\begin{split}
   |\phi_{f}(x[0,N-1], \omega)| \leq & 
     \sum \limits_{j=0}^m \left(|\phi_f(\beta^j(u_j)| +  |\phi_f(\beta^j(v_j)|  \right)  \\ \notag
    & + \|f \|_L \sum \limits_{j=0}^{m-1} H(\beta^{j}(u_j), \beta^{j+1}(u_{j+1}))   \\ \notag
    & + \|f \|_L  \sum \limits_{j=0}^{m-1} H(\beta^{m-j}(v_{m-j}) , \beta^{m-j-1}(v_{m-j-1})) \\ \notag
    & +  \|f \|_L H(\beta^m(u_m), \beta^m(v_m))  \notag 
    \end{split}
\end{equation}

If the rank of the cylindrical function $f$ is $n$ then the second and the third sums satisfy

\begin{equation}
     \| f\|_L \sum \limits_{j=0}^{m-1} |H \left( \beta^{j}(u_j), \beta^{j+1}(u_{j+1}) \right)| \leq Kmn \| f\|_L 
\label{eq7.1}
\end{equation}
for $K>0$.   We now only need to find a bound for the quantity

\begin{center}
   $\sum \limits_{j=0}^m \left(|\phi_f(\beta^j(u_j), \omega)| + |\phi_f(\beta^j(v_j), \omega)|  \right)$.
\end{center}

Each word of length $n$ is contained in exactly one cylinder of the corresponding length.  It is also true that the $H(v,w)$ collects at most $n$-many shifts of the concatenation of $v$ and $w$.  This implies that, in particular, we only collect at most $n$-many non-zero terms of the form $r_k 1_{[k,n]}(T^q(vw))$ from the $H(v,w)$.  Lemma \ref{lemma6.1} then implies that

\begin{equation} 
 |\phi_f(vw,\omega)| \leq |\phi_f(v, \omega)| + |\phi_f(w,\omega)|+ \| f\|_{\infty}n.
\end{equation}

Now set $L = \max \limits_{b \in A} |\beta(b)|$.  The above inequality and Lemma \ref{lemma7.2} implies 
\begin{equation}
 |\phi_f(\beta^j(u_j)| \leq L \left(\max \limits_{b \in  A} |\phi_{f}(\beta^j(b), \omega)| \right)+ n L \| f\|_{\infty} .
\end{equation}

Here we used the fact that $u_j$ and $v_j$ are prefixes and suffixes of the words $\beta(b)$ for $b \in A$.  Since we are considering a primitive substitution, the number of cylinders of length $n$, called $J_n$, is bounded above by  $C'n$ for some $C'>0$ dependent only on the substitution \cite{queffelec2010substitution}.   If we set

\begin{center}
    $\phi_f(v,\omega)= \sum_k r_k \phi_{[k,n]}(v,\omega)$.
\end{center}

then each $|r_k| \leq \| f\|_L$.  Thus for $j$ such that $|\beta^j(i)| \geq n$ we have

\begin{equation}
\begin{split}
 |\phi_{f}(\beta^j(b), \omega)| =  
 |\sum \limits_{k=1}^{J_n} r_k \phi_{[k,n]}(\beta^j(b),\omega)| \\ \label{eq7.4}
 &  \leq C'n \| f\|_L \max \limits_{k \leq J_n, i \in A} |\phi_{[k,n]}(\beta^j(i),\omega)|  \notag
 \end{split}
\end{equation}

We will now use a lemma that is analogous to what is found in \cite{bufetov2014modulus} for rank $n$ cylindrical functions.  \\

\begin{lemma}
  Let $\beta$ be a primitive substitution on $A$, and let $\theta$ be the Perron-Frobenius eigenvalue of the substitution matrix $S$.  Take a cylindrical function $f=\sum r_k 1_{[k,n]}$ and a number $\omega \in [0,1)$, and suppose there exists a sequence $\{F_{\omega}(n)\}_{n \geq 0}$ satisfying 
\begin{center}
    $\frac{F_{\omega}(n)}{\theta'} \leq F_{\omega}(n+1) \leq F_{\omega}(n)$, \quad $n \geq 0$
\end{center}
with $1 < \theta' < \theta$, such that 
\begin{center}
    $|\phi_f(\beta^m(b), \omega)| \leq n|\beta^m(b)| F_{\omega}(m)+mn$.
\end{center}

Then

\begin{center}
    $|\phi_f(x[0,N-1], \omega)| \leq \|f \|_L \left( \frac{C_1}{\theta-\theta'}nN F_{\omega}(\lfloor \log_{\theta}N-C_2 \rfloor)+6m^2n+1 \right)$
\end{center}
\end{lemma}

\begin{proof}
 The estimates (\ref{eq7.1}) - (\ref{eq7.4}) imply the first line in the inequality below, and the second line follows from the hypothesis in the lemma.  

\begin{equation}
\begin{split}
    |\phi_f(x[0,N-1], \omega)| & \leq n \| f\|_{L}\left[ K \sum \limits_{j=0}^{l}\max_{b \in A}|\beta^j(b)|F_{\omega}(j) +(2l+1) \right] \\
    & \leq nK \| f\|_L \left[ \sum \limits_{j=0}^{l} \theta^j (\theta')^{l-j}F_{\omega}(l) \right] \\
    & < K' \| f\|_L \left(  n   \theta^l F_{\omega}(l) + 6m^2n \right)
    \end{split}
\end{equation}
Since $c \theta^l \leq N \leq c' \theta^{l+1}$ the result follows.  
\end{proof}

We can now complete the proof for Theorem \ref{theorem7.1}. 

\begin{proof}
 If we consider both equation (\ref{eq7.4}) and Corollary \ref{corollary6.1} then 

\begin{equation}
    |\phi_f(\beta^j(b),\omega)|  \leq k'n |\beta^j(b)| \prod \limits_{k=n+1}^{m} (1-c' \| \omega  \|^2)+mn.
\end{equation}

Set $F_{\omega}(m)=\prod \limits_{k=n+1}^{m} (1-c' \| \omega  \|^2)$ and let  $c' < \frac{\theta-1}{\theta+1}$ (we can make $c'$ as small as needed).  If $\theta'=\frac{1+\theta}{2}$ then

\begin{equation}
    |\sum \limits_{j=0}^m \left(|\phi_f(\beta^j(u_j), \omega)| + |\phi_f(\beta^j(v_j), \omega)|  \right)| \leq 
     C_S' n \| f\|_L \left( N^{1-c' \| \omega\|^2} + 6m^2+1 \right)   
\end{equation}
The estimate on $m$ in the hypothesis completes the proof of Theorem \ref{theorem7.1}.  
 \end{proof}

\section{Speed of Ergodicity for Cylindrical Functions}
In order to estimate the quantity $I_{\epsilon}^{-}$ from equation (\ref{5.1}) we need to determine the behavior of the spectral measure $\sigma_f$ near $0$.  To accomplish this we are going to bound the Birkhoff sum of the cylindrical function $f$ and use the fact that bounds for this quantity imply bounds for the measure of small sets.   The main result of this section is the following theorem.  

\begin{theorem}
Let $f=\sum \limits_{k=1}^{J_n} r_k 1_{[k,n]}$ be a rank $n$ cylindrical function with 
\begin{center}
$0=\int_X f d \mu = \sum \limits_{k=1}^{J_n} r_k \mu([k,n])$ 
\end{center}
and let $\| f\|_{\infty}$ be the maximum of the absolute value of $f$.  Then the Birkhoff sums of $f$ have the following bound for some $C''>0$ dependent only on the substitution:

\begin{center}
$| \sum \limits_{j=0}^{N-1}f(T^j(x) | \leq n C''  \| f\|_{\infty}  (\log_{\theta}N)^2$.
\end{center}
\end{theorem}

\begin{proof}
We will need the following result, which is theorem 3 from \cite{adamczewski2004symbolic}.    

\begin{theorem}  Let $A^n(X)$ denote the allowable words of length $n$ in the substitution dynamical system $(X,T)$.  Then we have
\begin{center}
   $D_N(X)=\sup \limits_{x \in X} \sup \limits_{w \in A^n(X)}| \sum \limits_{k=0}^{N-1} \left( 1_{[w]}(T^k(x)-N \mu([w]) \right)| \leq C [\log_{\theta}N]^2$
   \end{center}
\label{speedergo}
\end{theorem}
uniformly in $n$ where $C$ depends only on the substitution.  \\
\qedsymbol
 \\
\-\
Recall that the number of words of length $n$ in $X$, denoted by $J_n$, satisfies $J_n \leq C'n$.  Since $f$ has zero average the lemma implies that 

\begin{equation*}
\begin{split}
    |\sum \limits_{j=0}^{N-1}f(T^jx)| 
   & =|\sum \limits_{j=0}^{N-1} \left(\sum_{k}r_k 1_{[k,n]}(T^jx)-Nr_k \mu([k,n]) \right)| \\
   & \leq J_n \| f\|_{\infty} \max \limits_{1 \leq k \leq J_n}|\sum_{j=0}^{N-1}1_{[k,n]}T^j(x)-N \mu([k,n])| \\
   & \leq C'n \| f\|_{\infty}D_N(X) \\
   & \leq C''n \| f\|_{\infty} [\log_{\theta}N]^2 
\end{split}
\end{equation*}

This implies that the speed of ergodicity for cylindrical functions with average zero of rank $n$ satisfies

\begin{center}
   $|\sum \limits_{k=0}^{N-1}f(T^kx)| \leq  n C'' \| f\|_{\infty}\left( \log_{\theta}N\right)^2$
\end{center}

where $C''$ is independent of $N$ and $x$ and depends only on the substitution.  

\end{proof}

\section{Speed of Weak Mixing for Particular Cylindrical Functions}
Some of the analytic techniques from this section are borrowed from \cite{avila2021quantitative}.  The framework has been slightly adapted for the case of cylindrical functions of arbitrarily large rank.  
\begin{theorem} If $f$ is a cylindrical function of rank $n = \lfloor (\log_{\theta}N)^{\frac{1}{6}} \rfloor$ that has zero average and if $g \in L^2(X)$, then
\begin{center}
    $\frac{1}{N}\sum \limits_{k=0}^{N-1} |\langle U^k(f), g \rangle |^2 \leq K_S \|f \|_L \|f \|_2 \| g\|_2^{2}[\log_{\theta} N  ]^{-\frac{1}{6}}$
\end{center}
\label{theorem9.1}
\end{theorem}

\begin{proof}
Observe that

\begin{equation*}
\begin{split}
    I= \sum \limits_{n=0}^{N-1} |\langle U^n(f), g \rangle |^2 
   & =\sum \limits_{n=0}^{N-1} \langle U^n(f), g \rangle \int_{R /Z} e^{2 \pi i n \omega } d \sigma_{f,g}( \omega) \\
   & = \int_{R / Z} \langle \sum \limits_{n=0}^{N-1} e^{2 \pi i n \omega} U^n(f), f \rangle d \sigma_{f,g}(\omega) \\
   & =I_{\epsilon}^{-}+I_{\epsilon}^{+}
   \end{split}
\end{equation*}

where 

\begin{center}
$I_{\epsilon}^{-}= \int \limits_{- \epsilon}^{\epsilon} \langle \sum \limits_{n=0}^{N-1} e^{2 \pi i n \omega} U^n(f), g \rangle d \sigma_{f,g}(\omega)$
\end{center}

and

\begin{center}
    $I_{\epsilon}^{+}= \int \limits_{ \epsilon}^{1-\epsilon} \langle \sum \limits_{n=0}^{N-1} e^{2 \pi i n \omega} U^n(f), g\rangle d \sigma_{f,g}(\omega)$.
\end{center}

We are now going to use bounds on the twisted sums to control $I_{\epsilon}^{+}$ and speed of ergodicity estimates for $I_{\epsilon}^{-}$.  The upper bound from Theorem \ref{theorem7.1} implies 

\begin{equation}
      |S_N^x(f, \omega)|   \leq   C_S' \cdot (\log_{\theta}N)^{\frac{1}{6}} \cdot \| f\|_L \left( N^{1-c' \epsilon^2} + (\log_{\theta}N)^2  \right)  .
\label{twistedbnd}     
\end{equation}

This along with the inequalities

\begin{center}
    $|\sigma_{f,g}|(B)| \leq \sqrt{ \sigma_{f,f}(B)}  \sqrt{\sigma_{g,g}(B)} \leq \| f\|_2 \| g\|_2$
\end{center}

for measureable $B$ implies that

\begin{align*}
    |I_{\epsilon}^{+}| \leq \|S_N^x(f, \omega)\|_{\infty} \| g\|_2^{2} \| f\|_2 \\
   & \leq C_s' (\log_{\theta}N)^{\frac{1}{6}} \left( N^{1-c' \epsilon^2}+(\log_{\theta}N)^2\right)\| f\|_L \| f\|_2 \| g\|_2^2
\end{align*}
We will now bound $I_{\epsilon}^{-}$.  We  need the following lemma from \cite{bufetov2014modulus} to relate the size of the measure of small $\epsilon$-balls with bounds for the twisted sums.  

\begin{lemma}
If $ G_N(f, \omega)= N^{-1} \int_X |S_N^x(f,\omega)|^2 d \mu(x)$ and $r=\lfloor (2N)^{-1} \rfloor$ then

\begin{center}
    $\sigma_{f}(B(\omega,r)) \leq \frac{ \pi^2}{4N}G_N(f, \omega) $
\end{center}

 \qedsymbol
\end{lemma}

In the following estimate we set $\omega = 0$ and use the above lemma along with equation (\ref{twistedbnd}) to get
\begin{center}
    $\sigma_{f}(B(0,\epsilon)) \leq K_1 \| f\|_L^2 (2 \epsilon)^2 [  \log_{\theta}(\frac{1}{2 \epsilon}) ]^{4+\frac{1}{3}}  $.
\end{center}

It follows that

\begin{align*}
I_{\epsilon}^{-} \leq \sigma_{f,g}(B(0, \epsilon)) N \| g\|_2 \|f \|_2\ \\
& \leq K_2(2 \epsilon)  [  \log_{\theta}(\frac{1}{2 \epsilon}) ]^{\frac{13}{6}} N  \|f\|_L\|f \|_2 \| g\|_2^{2} .
\end{align*}

To relate the upper bounds on $I_{\epsilon}^{+}$ and $I_{\epsilon}^{-}$, we will now write the quantity $N(2 \epsilon) [\log_{\theta}(\frac{1}{2 \epsilon}) ]^{\frac{13}{6}}$ as a power of $N$ and compare it with the quantity $N^{1-c \epsilon^2}$ on the right hand side of $I_{\epsilon}^{+}$. \\

If we let $u= \frac{1}{2 \epsilon}$ and  transform $u$ by the the increasing unbounded function

\begin{center}
    $u \to \frac{4u^2}{c'} \left( \ln (u) - \frac{13}{6} \ln \log_{\theta}u  \right)$
\end{center}

then for sufficiently large $N$ our choice of the new $u$ is equal to $\ln N$.  The powers are then equal and 

\begin{center}
    $I \leq C_s \|g \|_2^2 \| f\|_L \| f\|_2 \left( N^{1-c' \epsilon^2}+  (\log_{\theta}N)^{\frac{1}{6}} N^{1-c' \epsilon^2}+  (\log_{\theta}N)^{\frac{13}{6}} \right)$.
\end{center}

Since 

\begin{center}
    $ \frac{4}{c'} u^2 \ln u < \ln N < u^3$,
\end{center}

for large enough $u$ we have 

\begin{center}
    $N^{-\frac{c'}{4u^2}}< u^{-1} < [\ln N  ]^{-\frac{1}{3}}$.
\end{center}

Thus 

\begin{center}
    $I \leq K''\|f \|_2 \| g\|_2^{2} \|f \|_L \left(  [\ln N  ]^{-\frac{1}{3}} [ \log_{\theta}N]^{\frac{1}{6}}N+\frac{N}{(\log_{\theta}N)^{\frac{1}{3}}}\right) \leq K''' \frac{N}{(\log_{\theta}N)^{\frac{1}{3}}} \|f \|_2 \| g\|_2^{2} \|f \|_L $.
\end{center}

We have therefore proved that

\begin{center}
    $\frac{1}{N}\sum \limits_{k=0}^{N-1} |\langle U^k(f), g \rangle |^2 \leq K_S \|f \|_L \|f \|_2 \| g\|_2^2 [\log_{\theta} N  ]^{-\frac{1}{6}}$
\end{center}

where $K_S>0$ is a constant that depends only on the substitution.  This completes the proof of Theorem \ref{theorem9.1}.  

\end{proof}

\section{Speed of Weak Mixing on the Substitution Subshift}

A bounded function $f : X \to \mathbb{R}$ is weakly Lipschitz if there is a constant $C>0$ such that for any cylinder $[i,n]$, if $x, y \in [i,n]$ we have 
\begin{center}
    $|f(x)-f(y)| \leq C \mu([i,n])$.
\end{center}

Define $\| f\|_L=C_f+\| f\|_{\infty}$, where $C_f$ is the infimum of the $C$ as in the definition for weakly Lipschitz.   
We will find the speed of weak mixing for weakly Lipschitz functions by approximating their Birkhoff sums with rank $n$ cylindrical functions. 

\begin{lemma} If $f$ is weakly Lipschitz with zero average then there exists  cylindrical function $g_n$ of rank $n$ such that
\begin{center}
    $|f(y)-g_n(y)| = |f(y)-r_i|  \leq  C n^{-1} \|f \|_{L} $,
\end{center}
where $C > 0$ depends only on the substitution.  
\label{lemma10.1}
\end{lemma}

\begin{proof}

Define $g_n=\sum \limits_{k=1}^{J_n}r_k 1_{[k,n]}$ where $r_k= \frac{1}{\mu([k,n])} \int_{[k,n]}f d \mu$.  Since $f$ has zero average so does $g_n$.  Each $y \in X$ is contained in exactly one cylinder $y \in [i,n]$.  Thus, since $f$ is weakly Lipschitz, 

\begin{equation}
    |f(y)-g_n(y)| = |f(y)-r_i|  \leq \|f \|_{L} \mu([k,n]) .
\end{equation}

 There is some $C > 0$ such that $\mu([i,n]) \leq \frac{C}{n}$ uniformly in $n$ and $i$.  A proof for a similar lower bound can be found in Theorem 10.1, and the same technique can be used to establish this upper bound.  

\end{proof}

\begin{theorem}
If $f : X \to \mathbb{R}$ is weakly Lipschitz with zero average, and if $g \in L^2(X)$ then 

\begin{center}
    $\frac{1}{N}\sum \limits_{n=0}^{N-1} |\langle U^n(f), g \rangle |^2 \leq C_S \|f \|_L \| f\|_2\| g\|_2^2 [\log_{\theta} N  ]^{- \frac{1}{6}}$
\end{center}
where $C_S$ is dependent only on the substitution.  
\label{speedwm}
\end{theorem}

\begin{proof}
 Let $f_n$ be a cylindrical function that satisfies the condition from Lemma \ref{lemma10.1}.    Observe:

\begin{align*}
    N^{-1}\sum_{k=0}^{N-1}|\langle U^kf, g \rangle |^2  \\
  &  \leq N^{-1} \sum_{k=0}^{N-1}| \langle U^k(f-f_n), g \rangle|^2 +2| \langle U^k(f-f_n), g \rangle||\langle U^kf_n, g \rangle| +|\langle U^kf_n, g \rangle|^2 \\     
    & \leq 2C_S \frac{C}{n} \|f \|_L \| f\|_2 \|g \|_2^2  +  K_S  \frac{C}{n} \|f \|_L \|f \|_2 \| g\|_2^2+ \|f \|_L \|f \|_2 \| g\|_2^2[\log_{\theta} N  ]^{-\frac{1}{6}}.  
\end{align*}

The first two terms on the right hand side come from Lemma \ref{lemma10.1}, and the last one follows from Theorem \ref{theorem9.1}. 
Recall that in this case $n = \lfloor (\log_{\theta}N)^{\frac{1}{6}} \rfloor$.  Hence 

\begin{equation*}
    \sum \limits_{k=0}^{N-1}|\langle U^kf, g \rangle |^2\leq C_SN \|f \|_L \| f\|_2 \|g \|_2^2 [\log_{\theta}N ]^{-\frac{1}{6}}.
\end{equation*}

\end{proof}

\section{Speed of Weak Mixing for the Chacon map on the Interval}
The map that codes the orbit of the Chacon map into the substitution subshift has an inverse that transforms in a nice enough way that we can transfer the quantitative results from the Chacon substitution to the Chacon map on the interval.  This is the content of Theorem \ref{thm11.1}.

\begin{theorem}  The inverse of the code  $h: X \to I$ for the Chacon map is weakly Lipschitz.  
\label{thm11.1}
\end{theorem}

\begin{proof}
Let $x, y \in [i,n]$, and let $w$ be the length $n$ word such that $1_{[i,n]}(w)=1$.  Choose $m$ so that $3^{m-1} \leq n < 3^m$.  The fixed point of the substitution is generated by the inductive formula $S_{k+1}=S_k S_k 1 S_k$ with $S_0=0$.  Since the orbit of the fixed point is dense, and since the heights of the towers at the $k$'th step of the iteration is $h_k=\frac{3^{k+1}-1}{2}$, we have that $w$ is contained in the word $S_mS_m1S_m=S_{m+1}$.  We can see from the sequence $S_{m+1}$ that if $h(x)$ and $h(y)$ are at different vertical levels of the tower then the code for one of the strings $x$ or $y$ will hit the letter $1$ in $S_mS_m1S_m$ before the other in not more than twice the height of $S_m$ plus 1 steps in the orbit under the action of the Chacon map.  Since $3^{m-1} \leq n$, the points $h(x)$ and $h(y)$ must be in the same vertical level of the tower created during step $(m-2)$ of the construction.  Otherwise they would have a different code in less than $n$ steps, contradicting the fact that they are in a rank $n$ cylinder.  The length of the vertical levels  in the tower is equal to $l_k=\frac{2}{3^k}$.  By unique ergodicity, the measure of the cylinders of length $n$ is equal to its frequency in the fixed point of the substitution.  Since $n \leq 3^m$, and since the fixed point is dense, we can see that each block $S_{m+2}$ contains at least one instance of the word $w$.  Hence the frequency of $w$ is not less than the frequency of $S_{m+2}$.  We can see that there are $3$ occurrences of $S_{m+2}$ per string of length $h_{m+3}$. Hence the frequency of $S_{m+2}$ is bounded below by $\frac{C}{3^m}$.  Since $n \leq 3^m$, there is a $C'>0$ such that $l_{m-2} \leq C' \mu([i,n])$.  Thus

\begin{center}
    
    $|h(x)-h(y)| \leq l_{m-2} \leq C \mu([i,n])$.  
\end{center}

\end{proof}

\-\\
\begin{theorem}
The speed of weak mixing for the Chacon map on the unit interval is the same as that in Theorem \ref{speedwm}.  
\end{theorem}

\begin{proof}
  Let $C$ be the Chacon map on the interval and let $\sigma$ be the shift on the Chacon substitution.  The code $h$ satisfies $C \circ h = h \circ \sigma$.  If $\mu$ is the invariant measure on the uniquely ergodic substitution subshift, then the pullback measure $\mu(h^{-1}(A))$ is invariant under the action of the Chacon map.  Since $(I,C,L)$ is uniquely ergodic, $\mu(h^{-1}(A)=L(A)$.  It follows that for Lipschitz functions $f,g \in Lip(I)$,  
\begin{center}
    $\langle f \circ C^k , g \rangle_L=\langle f \circ C^k \circ h , g \circ h \rangle_{\mu}=\langle f \circ h \circ \sigma^k  , g \circ h \rangle_{\mu}$.
\end{center}
Since $f $ and $g$ are Lipschitz with $h$ weakly Lipschitz, $f \circ h$ and $g \circ h$ are both weakly Lipschitz.  If we use the formula

\begin{center}
    $|\sigma_{f,g}|(B) \leq \sqrt{ \sigma_{f,f}(B)}  \sqrt{\sigma_{g,g}(B)}$
\end{center}

along with Theorem \ref{speedwm} 

\begin{center}
    $\frac{1}{N}\sum \limits_{n=0}^{N-1} |\langle U^n(f), g \rangle |^2 \leq K_S \|f \|_L \|f \|_2 \| g\|_2[\log_{\theta} N  ]^{- \frac{1}{6}}$
\end{center}
where $K_S>0$ depends only on the substitution.  

\end{proof}
\section{Lower Bounds}
We will show that our upper bound for the speed of weak mixing is essentially sharp in that it cannot be improved to be better than logarithmic.  
Let $A_k = [0,2 \cdot 3^{-(k+1)}]$ and set $E_k= \{n \in \mathbf{N} : \mu(A_k \cap T^{-n}A_k)=0  \}$. From \cite{park2023size} we have the constraints  $C_k=(2 \cdot (h_k-3)! \cdot (\log(3)^{h_k})^{-1}$ and $t + 4 < h_k$ , where $h_k = (3^{k+1}-1)/2$.  Since $A_k = [0,2 \cdot 3^{-(k+1)}]$  we have $\mu(A_k) \geq \frac{1}{4h_k}$.  Now set $t=h_k-5$ and pick $N$ and $h_k$ so that $\frac{1}{2} \cdot \log(N) > h_k > 
\frac{1}{4} \cdot \log(N)$.  From Theorem 7.2 of \cite{park2023size} we have

\begin{lemma}
$E_k \cap [0,N] \geq C \log(N)^t$, where $C>0$ and $t > 0$.  
\end{lemma}

Let $g=1_{A_k}$ be the characteristic function on $A_k$. Now choose $f_{N}$ to be positive, continuously differentiable,  and supported on $A_{N,k}$ with the condition that
\begin{align}
\log(N)|\int_0^1 f_{N}(x) dx| \geq \frac{1}{100} \left(\max_{[0,1]}|f_N'|+\| f_N\|_{\infty} \right).
\end{align}
This is possible since $\log(N)$ is comparable to the measure of $\mu(A_k)^{-1}$ and $f_N$ is supported on  $A_k$.  Note that if $n \in E_k$ then $\int_0^1 f_N(T^nx)g dx = 0$.  This implies the following:

\begin{displaymath}
\begin{aligned}
\sum \limits_{i=0}^{N-1} \left|\int_x f(T^ix) g dx - \int_X f dx \int_X g dx\right| \\
&\geq \sum \limits_{i \in E_k} \left|\int_x f(T^ix) g dx - \int_X f dx \int_X g dx\right| \\
&= \sum \limits_{i \in E_k} \left| \int_X f dx \int_X g dx\right| \\
&\geq C_k \log(N)^t  \| g\|_2 \mu(A_k)^{\frac{1}{2}} \left|\int_X f dx\right|
\end{aligned}
\end{displaymath}
\label{eq: inequality1}

Here $t+4<h_k$ and $C_k=(2 \cdot (h_k-3)! \cdot (\log(3)^{h_k})^{-1}$ and $t + 4 < h_k$ , where $h_k = (3^{k+1}-1)/2$ \cite{park2023size}.  

\begin{lemma}
We can choose $t$ and $k$ so that
\begin{center}
$\cdot C_k \log(N)^t\mu(A_k)^{\frac{1}{2}} \geq C \cdot \frac{N}{\log(N)}$
\end{center}
for some $C>0$ independent of the parameters.
\end{lemma}

\begin{proof}

Sterling's formula gives the estimate $n! \leq en^{n+\frac{1}{2}}e^{-n}$.  Hence

\begin{align*}
    \log \left( C_k \log(N)^t \mu(A_k)^{\frac{1}{2}} \right) \\
    &\geq h_k-(h_k+\frac{1}{2})\log(h_k)-1 
 - h_k \log \log (3)  \\
 &+t \cdot \log \log(N)  - \frac{1}{2}\log(4 \cdot h_k).
\end{align*}

This is equivalent to 

\begin{center}
$O(h_k) + t \cdot \log \log(N) - h_k \log(h_k)$.
\end{center}

Since $\frac{1}{2} \cdot \log(N) > h_k > 
\frac{1}{4} \cdot \log(N)$, our choice of $t$ dictates that that the above quantity is bounded below by  

\begin{align}
C \left(\log(N) - \log \log (N) \right).
\end{align}

for some $C > 0$ independent of $k, N, f$ and $g$.  
\end{proof}

The Lipschitz norm of a function is the sum of its supremum norm and the smallest $C_f$ such that $|f(x)-f(y)| \leq C_f|x-y|$.  Hence Lemma 12.2 and equation (12.2) gives us 

\begin{align*}
\sum \limits_{i=0}^{N-1} \left|\int_X f(T^ix) g(x) dx - \int_X f(x) dx \int_X g(x) dx\right| \\
& \geq C \frac{N}{\log(N)^2} \| f_N\|_L \| g_{N}\| \\
& \geq  C \frac{N}{\log(N)^2} \cdot \| f_N\|_L^{\frac{1}{2}} \cdot \| f_N\|_2^{\frac{1}{2}} \cdot \| g_{N}\|_2
\end{align*}

with $g_{N}=1_{A_{N,k}}$ and $f_N$ chosen as above.

\section{Existence of 
Exceptional Sets}
Recall that a measure preserving transformation $T$ is weakly mixing if and only if for all measurable $A$ and $B$ there is an exceptional set $J_{A,B} \subset \mathbf{N}$ such that $\frac{1}{n} |J_{A,B} \cap [0.n] | \to 0$ and for $n \notin J_{A,B}$,

\begin{center}
$\lim \limits_{n \to \infty} \mu(T^{-n}(A) \cap B) \to \mu(A) \mu(B)$.
\end{center}

The following lemma from \cite{park2023size} gives us a way to bound the density of such an exceptional set $J_{f,g}$ for Lipschitz $f$ and square integrable $g$ in terms of the speed of weak mixing.  

\begin{lemma}
    Let $(a_n)$ be a decreasing sequence of non-negative numbers.  Suppose that 
\begin{center}
$\frac{1}{N} \sum \limits_{j=0}^{N-1} a_j \leq b_N$
\end{center}
for all $N \in \mathbf{N}$ and $b_N \to 0$ as $N \to \infty$.  Let $c_N$ decrease and converge to zero.  Then there is a set $J \subset \mathbf{N}$ such that $\frac{c_N}{N b_N} |[0, N] \cap J | $ converges to zero, and $a_n \to 0$ for $n \notin J$.  
\end{lemma}
\qedsymbol

In the case of the Chacon map we look to Theorem \ref{speedwm}.   If  we set $b_N = C \| f\|_2 \| f\|_L \| g\|_2^2 \log(N)]^{-\frac{1}{6}}$ and let $a_j = |\int_X f(T^jx)g(x) d \mu(x) |^2$ then lemma 12.1 with $c_N = C \| f\|_2 \| f\|_L \| g\|_2^2 [\log_3 (N)]^{-\frac{1}{3}}$ implies the following.  
\begin{theorem}
 Let $f$ be Lipschitz with zero average and $g \in L^2([0,1])$.  There is some $J_{f,g}$ such that $\langle f(T^k), g \rangle \to 0$ for $k \notin J_{f,g}$ and $\frac{1}{N} |J_{f,g} \cap [0,N]| \leq [\log_3(N)]^{-\frac{1}{6}}$
\end{theorem}

\newpage

\bibliographystyle{apalike}
\bibliography{references}

\end{document}